\documentclass[11pt,twoside, reqno]{amsart}

\usepackage{amsmath}
\usepackage{amsthm}
\usepackage{amsfonts, amssymb}
\usepackage{mathrsfs}
\usepackage[all]{xy}
\usepackage{url}

\usepackage{latexsym}
\usepackage[dvips]{graphicx}

\theoremstyle{plain}
\newtheorem{lema}{Lemma}
\newtheorem{prop}[lema]{Proposition}
\newtheorem{teo}[lema]{Theorem}

\theoremstyle{remark}

\newtheorem{obs}[lema]{Remark}

\theoremstyle{definition}

\newcommand{\cat}{\textrm{CAT}}

\newcommand{\N}{\mathbb{N}}
\newcommand{\R}{\mathbb{R}}
\newcommand{\cc}{\mathrm{c}}
\newcommand\restr{\raisebox{-0.3ex}{$|$}\raisebox{0.3ex}{}}
	
\begin{document}

\title[Lion and man in non-metric spaces]{Lion and man in non-metric spaces}

\author[J.A. Barmak]{Jonathan Ariel Barmak $^{\dagger}$}

\thanks{$^{\dagger}$ Researcher of CONICET. Partially supported by grants UBACyT 20020100300043, CONICET PIP 112-201101-00746 and ANPCyT PICT-2011-0812.}

\address{Universidad de Buenos Aires. Facultad de Ciencias Exactas y Naturales. Departamento de Matem\'atica. Buenos Aires, Argentina.}

\address{CONICET-Universidad de Buenos Aires. Instituto de Investigaciones Matem\'aticas Luis A. Santal\'o (IMAS). Buenos Aires, Argentina. }

\email{jbarmak@dm.uba.ar}

\begin{abstract}
A lion and a man move continuously in a space $X$. The aim of the lion is to capture his prey while the man wants to escape forever. Which of them has a strategy? This question has been studied for different metric domains. In this article we consider the case of general topological spaces.  
\end{abstract}



\maketitle

The original version of this problem is attributed to Rado who posed it 90 years ago: there is a lion and a man in a circular arena. They can move with the same maximum speed. The aim of the lion is to capture the man, while the man wants to escape. Which of them has a strategy? The players can look at each other at all times, each of them is considered as a point in the circle, and the lion captures the man only if the two points are in the exact same position. It is clear that if we replace the domain by the whole plane, the man can always escape, but what happens in the two dimensional disk $D^2$? Despite our intuition that the boundary should favor the lion, Besicovitch showed in 1952 that, even though the lion can get arbitrarily close, the man can escape forever by following certain polygonal. His beautiful argument is explained by Littlewood in \cite{Lit}.

What does it mean for the man to have a strategy in $D^2$? It means that for any path $\beta:[0,+\infty) \to D^2$ the lion follows with maximum speed $M$ (that is, $\beta$ is Lipschitz with Lipschitz constant $M$), the man will be able to follow another path $S(\beta) :[0,+\infty) \to D^2$ with maximum speed $M$ such that $S(\beta)(t)\neq \beta (t)$ for every $t\ge 0$. But since the man cannot take a look into the future, we require that his position at time $t$ is determined by the path $\beta \restr  _{[0,t)}$, meaning that if $\beta ':[0, +\infty)\to D^2$ is another path for the lion such that $\beta ' \restr _{[0,t)}=\beta \restr_{[0,t)}$, then $S(\beta') \restr_{[0,t]}=S(\beta) \restr_{[0,t]}$. A strategy for the lion is defined in a similar way. This natural interpretation of the word strategy was considered in \cite{BLW}. The paradoxical drawback of this definition is that there exist spaces in which both players have a strategy \cite[Theorems 8 and 9]{BLW}. Our belief that we can put both strategies to play one against the other is wrong: in order to determine the position of the man for some $t>0$ we need to know the position of the lion for each $t'<t$, but these in turn are determined by the positions of the man for $t'<t$. Since $[0,+\infty)$ is not well-ordered, this recursion does not yield a path for the man.

There are many variants of Rado's original problem. Some of them include pursuits in discrete time, in which the players take turns to play \cite{Sga}, other versions introduce the notion of capture radius \cite{ABG, BKK}, and the players are not regarded as points anymore. Some add restrictions to the paths \cite{Croft}, consider several pursuers chasing a unique evader or allow different maximum speeds. The regions where the game takes place go from subspaces of Euclidean spaces \cite{BKK2}, surfaces, $\cat(0)$ spaces \cite{ABG, BKK}, graphs and general metric spaces \cite{BLW}. The motivations of the problem vary from Robotics, Computer Science to Pure Mathematics.

In this paper we will study the lion and man problem in non-necessarily metric topological spaces. Of course, in this context the notion of speed makes no sense. The plethora of spaces in the non-Hausdorff zoo will make up for the lack of speed to produce interesting examples. The tools used in this article are completely elementary. 

Let $X$ be a topological space. Let $m,l\in X$ denote the starting positions of the man and the lion respectively. Given $x\in X$ denote by $P_x(X)$ the set of continuous maps $\gamma :[0,+\infty) \to X$ with starting point $\gamma (0)=x$. Given $\gamma \in P_x(X)$ and $t\in [0,+\infty)$, we denote by $\gamma _{<t}$ and $\gamma_{\le t}$ the restrictions $\gamma\restr_{[0,t)}$ and $\gamma\restr_{[0,t]}$. A \textit{strategy for the man} is a function $S:P_l(X)\to P_m(X)$ with the following properties:

\noindent i. For each $\beta \in P_l(X)$ and each $t\ge 0$, $S(\beta)(t)\neq \beta (t)$.

\noindent ii. If $\beta, \beta '\in P_l(X)$ and $t\ge 0$ are such that $\beta_{<t}=\beta'_{<t}$, then $S(\beta)_{\le t}=S(\beta)_{\le t}$.

The second requirement is known in \cite{BLW} as the \textit{no-lookahead rule}.

A \textit{strategy for the lion} is a function $S:P_m(X)\to P_l(X)$ with the following properties:

\noindent i. For each $\alpha \in P_m(X)$ there exists $t\ge 0$ such that $S(\alpha)(t)=\alpha (t)$.

\noindent ii. If $\alpha, \alpha '\in P_m(X)$ and $t\ge 0$ are such that $\alpha_{<t}=\alpha'_{<t}$, then $S(\alpha)_{\le t}=S(\alpha)_{\le t}$.
    
The following easy observation shows that a metric turns this game into a trivial pursuit.

\begin{prop} \label{leont2}
Let $X$ be a path-connected Hausdorff space and let $m,l\in X$. Then the lion has a strategy.
\end{prop}   
\begin{proof}
Define $S:P_m(X)\to P_l(X)$ as follows. Choose first any path $\gamma$ from $l$ to $m$. Let $\alpha \in P_m(X)$. Let $S(\alpha)(t)=\gamma (2t)$ for $t\le \frac{1}{2}$, $S(\alpha)(t)=\alpha(2t-1)$ for $\frac{1}{2}\le t\le 1$ and $S(\alpha) (t)=\alpha (1)$ for $t\ge 1$. Then $S(\alpha)\in P_l(X)$ and $S(\alpha)(1)=\alpha (1)$. The map $S$ satisfies the no-lookahead rule since $2t-1<t$ for $t<1$ and $S(\alpha)(1)=\alpha (1)$ is determined by $\alpha _{<1}$ by the Hausdorff hypothesis.  
\end{proof}

\begin{prop} \label{hombret2}
Let $X$ be a Hausdorff space which admits a fixed-point-free map $f:X\to X$. Then the man has a strategy for some $m,l\in X$.
\end{prop}
\begin{proof}
Let $l\in X$ be any point and take $m=f(l)$. Define $S:P_l(X)\to P_m(X)$ by $S(\beta)(t)=f(\beta(t))$. Once again, the Hausdorff axiom guarantees that $S$ satisfies the no-lookahead rule.
\end{proof}

For example in $S^1$, for antipodal (or any two different) starting points, both the lion and the man have strategies.

We turn now to non-Hausdorff spaces.

\begin{prop} \label{indiscreto}
Let $X$ be an indiscrete space. Then the lion has a strategy.
\end{prop}
\begin{proof}
Let $m,l\in X$. Define in $P_m(X)$ the following relation: $\alpha \sim \alpha '$ if there exists $t>0$ such that $\alpha _{<t}=\alpha'_{<t}$. This is clearly an equivalence relation. Denote by $\overline{\alpha}$ the class (germ) of $\alpha$. With the Axiom of Choice we choose for each class $c$ a representative $r(c)\in P_m(X)$. Define $S:P_m(X)\to P_l(X)$ by $S(\alpha)(0)=l$ and $S(\alpha)(t)=r(\overline{\alpha})(t)$ for $t>0$. Since $X$ is indiscrete, $S(\alpha)$ is continuous. If $\alpha_{<t}=\alpha'_{<t}$ for some $t>0$, then $\overline{\alpha}=\overline{\alpha'}$, so the no-lookahead rule is fulfilled. Moreover, $S(\alpha)$ and $\alpha$ coincide in an interval $(0,t)$.
\end{proof}

It is easy to find spaces where the man does not have a strategy. For instance if in $X=[0,1]$ the lion moves in a path $\beta$ that passes through $0$ and $1$, then any other path will coincide with $\beta$ for some $t\ge 0$. The previous results seem to give evidence that the lion has a strategy in every space. However, we will show that the Axiom of Choice can be against the lion in some examples.

\begin{teo}\label{main}
There exists a space $X$ in which the lion does not have a strategy for some initial points $m,l\in X$. 
\end{teo}
\begin{proof}
We imitate the classical construction of the uncountable well-ordered set in which each proper section is countable (\cite[Lemma 10.2, Theorem 10.3]{Mun}). Let $\cc$ be the cardinality of the continuum $\mathbb{R}$. We construct a non-empty well-ordered set $X$ without maximum such that every subset $Y\subseteq X$ with cardinality $\# Y\le \cc$ is bounded above. For this take any well-ordered set $Z$ with maximum and cardinality greater than $\cc$ and consider the smallest element $z\in Z$ such that the section $Z_{<z}$ has cardinality greater than $\cc$. Then $X=Z_{<z}$ does not have a maximum and if $Y\subseteq X$ is such that $\# Y\le \cc$, then $\#(\bigcup\limits_{y\in Y} X_{<y})\le \cc.\cc=\cc$ so $\bigcup\limits_{y\in Y} X_{<y}\neq X$ and then $Y$ has an upper bound in $X$. 

Now consider the \textit{Alexandroff topology} in $X$ whose proper open sets are the sections $X_{<x}$ for $x\in X$ (\cite{Ale}). This is not the usual order topology (in which a basis is given by intervals $(a,b)$). $X$ is path-connected since the partial order in $X$ is a total order. Concretely, given $x\le y \in X$, $\gamma (t)=x$ for $t<1$ and $\gamma (1)=y$ defines a path from $x$ to $y$ (cf. \cite{Mcc, Sto}).

Let $l$ be the minimum of $X$ and $m$ the second element of $X$. Suppose that $S:P_{m}(X)\to P_l(X)$ is a strategy for the lion. 

If we take $\alpha_0 \in P_m(X)$ to be the constant map $m$, then $S(\alpha_0)\in P_l(X)$. Since $\{l\}=X_{<m}\subseteq X$ is open, there is an interval $[0,t_0)$ in which $S(\alpha_0)$ is smaller than $m$ (constant $l$). Let $l_1=S(\alpha_0)(t_0)$ and let $m_1>l_1$. Redefine $\alpha_0$ for $t\ge t_0$ as follows: Let $\alpha _1\in P_{m}(X)$ coincide with $\alpha_0$ in $[
0,t_0)$ and be constant $m_1$ for $t\ge t_0$. By the no-lookahead rule, $S(\alpha_1)(t_0)=l_1$. Moreover, since $X_{< m_1} $ is open, there is an interval $[t_0, t_1)$ in which $S(\alpha_1)$ is smaller than $m_1$ and so $S(\alpha_1)$ and $\alpha _1$ do not coincide at any $t\in [0,t_1)$. Be repeating this idea, we can push the $t_i$ further away. In order to formalize this, we will define an order in certain subset of $P_m(X)$, prove the existence of a maximal element $\mu$ using Zorn's Lemma and get a contradiction by finding a greater element.

Let $A$ be the set of those $\alpha \in P_m(X)$ for which there exists $t_{\alpha} \in [0,+\infty)$ with the following properties

\begin{itemize}
\item $ \alpha (t)<\alpha (t_{\alpha})$ for every $t<t_\alpha$,
\item $\alpha (t)=\alpha (t_\alpha)$ for every $t\ge t_\alpha$,
\item $S(\alpha)(t)\neq \alpha (t)$ for each $t<t_\alpha$.
\end{itemize}

Note that the element $t_\alpha$ is uniquely determined by $\alpha \in A$. Also $A\neq \emptyset$ since the constant map $m$ is in $A$ ($t_\alpha=0$).

We define an order in $A$ as follows: $\alpha \lhd \alpha '$ if $t_\alpha \le t_{\alpha'}$ and $\alpha _{\le t_\alpha}=\alpha '_{\le t_\alpha}$. Clearly $\lhd$ is reflexive, transitive and antisymmetric.

Let $C=\{\alpha _i\}_{i\in I}$ be a chain in $A$. We want to prove it has an upper bound. Define $\alpha \in P_m(X)$ as follows. Given $t\ge 0$, if there exists $i\in I$ with $t_{\alpha_i}>t$, define $\alpha (t)=\alpha_i(t)$. This is well-defined since $C$ is a chain. Moreover, $\alpha$ is continuous in $[0,t_{\alpha_i})$ for every $i\in I$. If the set $\{t_{\alpha_i}\}_{i\in I}\subseteq [0,+\infty)$ is unbounded, $\alpha$ is defined and continuous in all $[0,+\infty)$. Moreover, since $\alpha _{<t_i}=(\alpha_i)_{<t_i}$, then $S(\alpha)_{\le t_i}=S(\alpha_i)_{\le t_i}$ and therefore $S(\alpha)(t)\neq \alpha (t)$ for every $t<t_i$. Hence $S(\alpha)(t)\neq \alpha (t)$ for every $t\in [0, +\infty)$, contradicting the definition of strategy. Thus, $\{t_{\alpha_i}\}_{i\in I}$ is bounded. Let $T= \sup \{t_{\alpha_i}\}_{i\in I}$. If $T=t_{\alpha_i}$ for some $i$, then $\alpha _i$ is an upper bound for $C$. Assume then that $T>t_{\alpha_i}$ for every $i$. The set $Y=\{\alpha_i(t_{\alpha_i})\}\subseteq X$ has cardinality at most $\cc$ since $\{t_{\alpha_i}\}_{i\in I}\subseteq \mathbb{R}$ and $t_{\alpha_i}=t_{\alpha_j}$ implies $\alpha_i=\alpha_j$. Therefore $Y$ has an upper bound $x_\alpha\in X$. Recall that $\alpha$ was already defined and continuous in each interval $[0,t_{\alpha_i})$ and then in $[0,T)$. Define $\alpha (t)=x_\alpha$ for each $t\ge T$. We claim that $\alpha \in A$ and that it is an upper bound for $C$.

We will prove that $\alpha$ is continuous. Note first that if $t<T$, there exists $i$ with $t_{\alpha_i}>t$, so $\alpha (t)=\alpha_i(t) < \alpha_i(t_{\alpha_i})\le x_\alpha$. Now take a proper open set of $X$, that is a section $X_{<x}$.  If $x>x_\alpha$, $\alpha ^{-1}(X_{<x})=[0,+\infty)$ is open. If $x\le x_\alpha$, $\alpha ^{-1}(X_{<x})=\alpha_{<T}^{-1}(X_{<x})\subseteq [0,+\infty)$ is open since $\alpha_{<T}$ is continuous and $[0,T)$ is open. It is easy to see then that $\alpha \in A$ with $t_\alpha=T$ and that $\alpha_i \lhd \alpha$ for every $i\in I$.

By Zorn's Lemma, $A$ has a maximal element $\mu$. Now we apply the idea of the beginning to push $t_\mu$ further away. Let $x=S(\mu)(t_\mu)\in X$. Since $X$ does not have a maximum, there exists $y\in X$ such that $y>\max \{\mu (t_\mu), x\}$. Define $\mu' \in P_m(X)$ in the following way. $\mu '(t)=\mu (t)$ for $t<t_\mu$ and $\mu'(t)=y$ for $t\ge t_\mu$. Since $\mu '_{<t_{\mu}}=\mu _{<t_{\mu}}$, $S(\mu ')_{\le t_\mu}=S(\mu)_{\le t_\mu}$ and then $S(\mu ')$ and $\mu '$ do not coincide in $[0,t_\mu]$. Since $X_{<y}$ is open, $S(\mu')^{-1}(X_{<y})\subseteq [0,+\infty)$ is an open set which contains $t_{\mu}$. Therefore, there exists $t'>t_{\mu}$ such that $S(\mu ')(t)\neq \mu '(t)$ for every $t<t'$. Finally, define $\nu \in P_m(X)$ to be equal to $\mu '$ for $t<t'$ and constant $x_\nu$ for $t\ge t'$, where $x_\nu \in X$ is any element greater than $y$. Then $\nu \in A$, with $t_\nu=t'$, and it is strictly greater than $\mu$, a contradiction.
\end{proof}  

\begin{obs}
In the space $X$ constructed in Theorem \ref{main}, the man does not have a strategy, independently of the starting points. Suppose $S:P_l(X)\to P_m(X)$ is a strategy. Let $\beta \in P_l(X)$ be a path that goes from $l$ to $x_0=\min (X)$ in the interval $[0,\frac{1}{2}]$ and is constant $x_0$ for $t\ge \frac{1}{2}$. Define $\beta' \in P_l(X)$ to equal to $\beta$ in $[0,1)$ and constant $S(\beta)(1)$ in $[1,+\infty)$. This map is continuous. However, $S(\beta')(1)=S(\beta)(1)=\beta ' (1)$, a contradiction. Therefore, no player has a strategy in this space for a particular choice of the starting points.

Of course, the same argument shows that the man does not have a strategy in any indiscrete space.
\end{obs}

\bigskip

One important class of non-Hausdorff spaces is given by \textit{$A$-spaces}, which model up to weak homotopy equivalence every CW-complex (see \cite{Mcc} and \cite{Bar}). An $A$-space is a space in which arbitrary intersections of open sets are open. For instance, any finite topological space is an $A$-space. If $X$ is an $A$-space, the \textit{open hull} of a subspace $Y$ is the intersection of all the open sets containing $Y$. The open hull of a point $x\in X$ will be denoted by $U_x$. Clearly the open hull of $Y\subseteq X$ is $\bigcup\limits_{y\in Y} U_y$. Given an $A$-space $X$ its \textit{dual} is the $A$-space with the same underlying set, but whose open sets are the closed sets of $X$. Therefore, the dual of a separable $A$-space is an $A$-space $X$ which contains a countable subset $Y$ whose open hull is $X$. Any indiscrete space satisfies this property. Thus the following result is a generalization of Proposition \ref{indiscreto}.

\begin{teo}
Let $X$ be a path-connected $A$-space which contains a countable subset $Y$ whose open hull is $X$. Then for any starting points the lion has a strategy. In particular in any finite space the lion has a strategy.
\end{teo}
\begin{proof}
Let $m,l\in X$ be the starting points. Let $(y_n)_{n\in \N}$ be a sequence of points of $Y$ such that every $y\in Y$ appears infinitely many times in the sequence. Let $(t_n)_{n\in \N}$ be a strictly increasing sequence in $(0,1)$ which converges to $1$. Since $X$ is path-connected there exists a continuous curve $\beta: [0,1)\to X$ starting in $l$ such that $\beta (t_n)=y_n$ for every $n\in \N$. Let $\alpha \in P_m(X)$. We claim that there exists $n\in \N$ such that $\alpha (t_n) \in U_{y_n}$. Indeed, since $\lim\limits_{n\to \infty} \alpha (t_n) =\alpha (1)$, then $\alpha (t_n) \in U_{\alpha (1)}$ for $n$ big enough. Since $X=\bigcup\limits_{y\in Y} U_y$, $\alpha (1)\in U_y$ for some $y\in Y$ and then $U_{\alpha (1)}\subseteq U_y$. Take $n$ big enough and such that $y_n=y$. Then $\alpha (t_n) \in U_{y_n}$. For each $\alpha \in P_m(X)$ define $n_{\alpha}$ to be the smallest $n\in \N$ such that $\alpha (t_n) \in U_{y_n}$.

We define now an equivalence relation in $P_m(X)$. Say that $\alpha \sim \alpha'$ if $n_\alpha=n_{\alpha '}$ and there exists $t>t_{n_\alpha}$ such that $\alpha \restr_{(t_{n_\alpha},t)}=\alpha ' \restr_{(t_{n_\alpha},t)}$. With the Axiom of Choice take for every equivalence class $c$ a representative $r(c)$. Given $\alpha \in P_m(X)$ define $S(\alpha) \in P_l(X)$ as follows. $S(\alpha)(t)=\beta (t)$ for $t\le t_{n_\alpha}$ and $S(\alpha)(t)=r(\overline{\alpha})(t)$ for $t>t_{n_\alpha}$. The path $S(\alpha)$ is continuous since $r(\overline{\alpha})(t_{n_\alpha}) \in U_{\beta (t_{n_\alpha})}$. Moreover, $S:P_m(X) \to P_l(X)$ satisfies the no-lookahead rule. If $\alpha_{<t}=\alpha '_{<t}$ for some $t\in [0,+\infty)$ and $t_{n_\alpha} \ge t$, then $S(\alpha)$ and $S(\alpha ')$ coincide with $\beta$ in $[0,t]$. If $t_{n_\alpha}<t$, then $\alpha \sim \alpha '$ so $S(\alpha)=S(\alpha ')$. Finally, for any $\alpha \in P_m(X)$, the paths $S(\alpha)$ and $\alpha$ coincide in $(t_{n_\alpha},t)$ for some $t>t_{n_\alpha}$. This shows that $S$ is a strategy for the lion.    
\end{proof}

So far we have constructed spaces in which both players have a strategy (like $S^1$), where none of them has a strategy (the space in Theorem \ref{main}) or where only the lion has a strategy (any indiscrete space). We will show now that there are spaces where the man is the only player with a strategy.

\begin{lema} \label{retracto}
Let $X$ be a topological space and let $r:X\to A$ be a retraction onto a subspace $A\subseteq X$. Let $m\in A$ and $l\in X$. Then

\noindent (a) If the man has a strategy in $A$ for initial points $m, r(l) \in A$, then it also has a strategy in $X$ for initial points $m,l\in X$.

\noindent (b) If the lion has a strategy in $X$ for initial points $m,l\in X$, then it also has a strategy in $A$ for initial points $m, r(l) \in A$.
\end{lema}
\begin{proof}
Let $S:P_{r(l)}(A)\to P_m(A)$ be a strategy for the man in $A$. Define $\widetilde{S}:P_l(X) \to P_m(X)$ by $\widetilde{S}(\beta)=S(r\beta)$. Clearly $\widetilde{S}$ satisfies the no-lookahead rule. We must check that $\widetilde{S}(\beta)(t)=S(r\beta)(t)\neq \beta (t)$ for every $t$. If $\beta (t)\notin A$, this is obvios since $S(r\beta)(t)\in A$. If $\beta(t)\in A$, then $S(r\beta)(t)\neq r\beta(t)=\beta (t)$.

For the second part, suppose $S:P_m(X)\to P_l(X)$ is a strategy for the lion. It is easy to check that the map $\overline{S}:P_m(A)\to P_{r(l)}(A)$ defined by $\overline{S}(\alpha)(t)=r(S(\alpha)(t))$ provides a strategy for the lion in $A$.
\end{proof} 

\begin{prop}
There exists a path-connected space $Y$ and starting points $m,l\in Y$ for which only the man has a strategy. 
\end{prop}
\begin{proof}
Let $X$ be the space constructed in Theorem \ref{main}, $x_0\in X$ its minimum and $x_1\in X$ its second element. Recall than the lion does not have a strategy if he starts at $x_0$ and the man starts at $x_1$. Let $Y=X\vee S^1 \vee X$ be the space obtained from two copies $X$ and $X'$ of $X$ and one of $S^1$ by identifying the point $x_1\in X$ with a point $p\in S^1$ and the corresponding point $x'_1\in X'$ with the antipodal point $q\in S^1$ of $p$. Define $m=x_1=p$ and $l=x'_0$. There is a retraction $Y\to X$ which maps $X'$ to $X$ with the identity and maps $S^1$ to $m$. By the previous lemma, the lion does not have a strategy in $Y$ for those starting points. There is another retraction $Y\to S^1$ which maps $X$ to $m$ and $X'$ to $x'_1=q$. By the lemma, the man has a strategy in $Y$.   
\end{proof}

To finish, we go back to the start. Recall that in Rado's original problem lion and man moved in a circular arena with the same maximum speed. For 25 years the lion was believed to be the one with a strategy. This ``strategy'' consisted in keeping the lion in the radius determined by the man. The argument used the wrong assumption that the best thing for the man was to stay in the boundary of the circle. Besicovitch showed that the man can always escape from the lion but staying in the interior of the arena. Proposition \ref{leont2} says in particular that in our version of the problem, the lion has a strategy in $D^2$. What about the man? We cannot use Proposition \ref{hombret2} by the Brouwer Fixed Point Theorem. In fact, \cite[Theorem 7]{BLW} shows that there is no continuous strategy for the man (considering $P_l(D^2)$ with a topology that makes the inclusion $i:D^2\to P_l(D^2)$ continuous, where $i(x)$ is the straight path from $l$ to $x$, and considering $P_m(D^2)$ with any topology that makes the evaluation $ev:P_m(D^2)\to D^2$ in $1$ continuous, like the compact-open topology). Already escaping from a unique lion $\beta \in P_l(D^2)$ does not seem simple. How can the man escape a Peano lion which fills the disk?

We will prove that given any curve $\beta \in P_0(D^2)$, there exists a path $S(\beta) \in P_1(D^2)$ which escapes from $\beta$. Moreover, we will prove that $S$ can be constructed satisfying the no-lookahead rule. In other words, Besicovitch's result also holds in our setting. Surprisingly enough, in our strategy the man stays all the time in the boundary of the disk.

We recall for the non-expert the statement of a very particular case of the lifting lemma \cite[Lemma 79.1]{Mun}. Suppose that $J\subseteq \R$ is an interval and that $f:J\to S^1$ is a continuous map. Then there exists a lifting of $f$ to the universal cover of $S^1$, that is a continuous map $\widetilde{f}:J\to \R$ such that $p\widetilde{f}=f$, where $p:\R \to S^1 \subseteq \mathbb{C}$ is defined by $p(t)=e^{2\pi i t}$. Moreover, if $x_0\in J$ and $t_0\in \R$ are such that $f(x_0)=p(t_0)$, then there exists a unique lifting $\widetilde{f}$ of $f$ such that $\widetilde{f}(x_0)=t_0$.

\begin{teo} \label{besi}
Let $l=0\in D^2\subseteq \mathbb{C}$ be the center of the disk and let $m=1\in D^2$. Then the man has a strategy. Moreover, he can keep on the boundary of the disk during the whole pursuit.
\end{teo}
\begin{proof}
The idea is the following. While the lion is inside the concentric circle of radius $\frac{1}{2}$, the man stays in the point $m$. When the lion goes outside that circle, the man moves continuously in $S^1$ in such a way that, when the lion reaches the boundary of $D^2$, the man is in the antipodal point. If $\beta \in P_l(D^2)$, then $\beta (t)=\rho(t)e^{2\pi i \omega (t)}$ for some $\rho:[0,+\infty)\to \R_{\ge 0}$ continuous and some $\omega: [0,+\infty) \to \R$ continuous in $\beta ^{-1}(D^2\smallsetminus \{0\})$. We define $\alpha \in P_m(D^2)$ as follows: $\alpha(t)=e^{2\pi i \theta (t)}$ with $\theta (t)=0$ if $\rho (t)\le \frac{1}{2}$ and $\theta (t)=(2\rho (t)-1)(\omega (t) +\frac{1}{2})$ if $\rho (t) \ge \frac{1}{2}$. This is well defined and continuous, and allows the man to escape the curve $\beta$. However, to turn this into a strategy for the man we must show that $\omega$ can be chosen satisfying the no-lookahead rule.

Given $a\in [0,+\infty)$, define an equivalence relation in the set of continuous maps with values in $S^1$ defined in an interval $(a,b)$ for some $b>a$. We say that $f:(a,b)\to S^1$ and $g:(a,c)\to S^1$ are \textit{$a$-equivalent} if there exists $t\in (a,b)\cap (a,c)$ such that $f\restr _{(a,t)}=g\restr _{(a,t)}$. For each $a$ and each equivalence class choose a representative $f:(a,b)\to S^1$ and choose a lifting $\widetilde{f}:(a,b)\to \R$ to the universal cover of $S^1$.

Now, let $\beta \in P_l(D^2)$. Let $r:D^2\smallsetminus \{0\} \to S^1$ be the radial (or any other continuous) retraction. Then the composition $r\beta \restr _{\beta ^{-1}(D^2\smallsetminus \{0\})} : \beta ^{-1}(D^2\smallsetminus \{0\})\to S^1$ is well-defined and continuous. Since $\beta ^{-1}(D^2\smallsetminus \{0\})\subseteq [0,+\infty)$ is open and does not contain the origin $0$, it is a disjoint union of open intervals $(a_i,b_i)$. We can lift $r\beta \restr _{(a_i,b_i)}$ to a map $(a_i,b_i)\to \R$ by considering the representative of the $a_i$-class and its chosen lifting. Concretely, if the chosen representative of $r\beta \restr _{(a_i,b_i)}$ is $f:(a_i,c)\to S^1$ and $r\beta \restr _{(a_i,t)}=f\restr _{(a_i,t)}$ for some $t\in (a_i,b_i) \cap (a_i,c)$, then the chosen lifting $\widetilde{f}:(a_i,c)\to \R$ restricts to a lifting of $f\restr _{(a_i,t)}$ and this extends to a lifting of $r\beta \restr _{(a_i,b_i)}$. The family of maps $(a_i,b_i) \to \R$ determines a continuous lifting $\omega :\beta ^{-1}(D^2\smallsetminus \{0\})\to \R$ of $r\beta \restr _{\beta ^{-1}(D^2\smallsetminus \{0\})}$. Now define $S(\beta) \in P_m(D^2)$ as explained above: $S(\beta)(t)=e^{2\pi i \theta (t)}$ where $\theta(t)=0$ if the norm $\left\lVert \beta (t) \right\rVert \le \frac{1}{2}$ and $\theta (t)=(2\left\lVert \beta (t) \right\rVert-1)(\omega (t) +\frac{1}{2})$ if $\left\lVert \beta (t) \right\rVert \ge \frac{1}{2}$. Then $S(\beta)$ is continuous. Moreover, if $\left\lVert \beta (t) \right\rVert=1$, then $\theta (t)=\omega (t)+\frac{1}{2}$, so $S(\beta)(t)=-e^{2\pi i \omega (t)}=-r\beta (t)=-\beta (t)$. In particular $S(\beta)(t)\neq \beta (t)$ for every $t\ge 0$. We verify the no-lookahead rule. Suppose $\beta_{<t}=\beta'_{<t}$ for some $\beta, \beta'\in P_l(D^2)$ and $t> 0$. If $\beta(t)=\beta'(t)=0$, then $t$ does not belong to any of the intervals $(a_i,b_i)$ in the decomposition of $\beta ^{-1}(D^2\smallsetminus \{0\})$ nor the intervals $(a'_j,b'_j)$ in the decomposition of $(\beta ') ^{-1}(D^2\smallsetminus \{0\})$. Moreover, the intervals $(a_i,b_i)$ with $b_i<t$ and the intervals $(a'_j,b'_j)$ with $b'_j<t$ are the same, so $\omega : \beta ^{-1}(D^2\smallsetminus \{0\}) \cap [0,t) \to \R$ coincides with $\omega ': (\beta ') ^{-1}(D^2\smallsetminus \{0\}) \cap [0,t) \to \R$ and then $S(\beta)_{\le t}=S(\beta ')_{\le t}$. Suppose now that $\beta (t)=\beta '(t)\neq 0$, then $t$ belongs to an interval $(a_i,b_i)$ and to another $(a'_j,b'_j)$ with $a_i=a'_j$. The intervals at the left of $a_i$ are the same in both decompositions. By assumption $r\beta \restr _{(a_i,b_i)}$ and $r\beta' \restr _{(a'_j,b'_j)}$ are $a_i$-equivalent. Then the liftings $(a_i,b_i)\to \R$ and $(a'_j,b'_j)\to \R$ of both maps coincide in an interval $(a_i,t')$, and therefore also in $(a_i,t]$. Thus $\omega= \omega ' :\beta ^{-1}(D^2\smallsetminus \{0\}) \cap [0,t]$ and then $S(\beta)_{\le t}=S(\beta')_{\le t}$.

\end{proof}

Theorem \ref{besi} can be used together with Lemma \ref{retracto} to show that the man has a strategy in a large class of examples. Suppose $X$ is a normal space and $l\in X$ is such that there is a subspace $U\ni l$ of $X$ homeomorphic to an open 2-dimensional disk. Since $D^2$ is an absolute retract, for any $m\in X$ different from $l$, the man will have a strategy. This can be applied for instance to any $n$-dimensional manifold with $n\ge 2$.  

\bigskip

\textbf{Acknowledgment:} I want to thank Charly Di Fiore for showing me how to use AC to save prisoners with colored hats many years ago.

\end{document}